\documentclass[12pt]{amsart}
\usepackage{format}

\title{Restricting Representations from a complex reductive group to a real form}

\begin{document}
\maketitle

\begin{abstract}
    Let $G$ be a complex connected reductive algebraic group and let $G_{\RR}$ be a real form of $G$. We construct a sequence of functors $L_i\mathcal{R}$ from admissible (resp. finite-length) representations of $G$ to admissible (resp. finite-length) representations of $G_{\RR}$. We establish many basic properties of these functors, including their behavior with respect to infinitesimal character, associated variety, and restriction to a maximal compact subgroup. We deduce that each $L_i\mathcal{R}$ takes unipotent representations of $G$ to unipotent representations of $G_{\RR}$. Taking the alternating sum of $L_i\mathcal{R}$, we get a well-defined homomorphism on the level of characters. We compute this homomorphism in the case when $G_{\RR}$ is split. 
\end{abstract}

\section{Introduction}

Let $G$ be a complex connected reductive algebraic group and let $G_{\RR}$ be a real form of $G$. Write $M_{fl}(G)$ (resp. $M_{fl}(G_{\RR})$) for the category of finite-length admissible representations of $G$ (resp. $G_{\RR}$). There are many difficult questions about admissible representations of real reductive groups which have (relatively) easy answers in the case of complex groups. One example is the classification of unipotent representations. Thus, it is natural to look for a relationship between the categories $M_{fl}(G)$ and $M_{fl}(G_{\RR})$. One such relationship is \emph{base change} (see \cite[Chapter 1.7]{ArthurClozel}). This is a map
$$\mathrm{BC}: KM_{fl}(G_{\RR}) \to KM_{fl}(G)$$
from the Grothendieck group of $M_{fl}(G_{\RR})$ to the Grothendieck group of $M_{fl}(G)$. In Section \ref{sec:functors}, we will define a homomorphism in the \emph{opposite} direction. In fact, we will define a sequence of functors
\begin{equation}\label{eq:LiRintro}L_i\mathcal{R}: M_{fl}(G) \to M_{fl}(G_{\RR}), \qquad i \geq 0.\end{equation}
These functors have nice behavior with respect to several important invariants, including infinitesimal character and associated variety (see Lemma \ref{lem:propsofR} and Proposition \ref{prop:propsofR}).

Taking the alternating sum of $L_iR$, we get a group homomorphism
\begin{equation}\label{eq:Rintro}\mathcal{R}: KM_{fl}(G) \to KM_{fl}(G_{\RR}), \qquad \mathcal{R}[X] = \sum_i (-1)^i [L_i\mathcal{R}(X)]\end{equation}
This homomorphism has even nicer behavior than the individual functors. In particular, the restriction of $\mathcal{R}[X]$ to a maximal compact subgroup is given by a relatively simple formula (see Corollary \ref{cor:Ktypes}).

In Section \ref{sec:split}, we will compute the homomorphism (\ref{eq:Rintro}), in the case when $G_{\RR}$ is split, in terms of the basis of standard representations.

In Section \ref{sec:unipotent}, we will show that the functors $L_i\mathcal{R}$ take unipotent representations of $G$ to unipotent representations of $G_{\RR}$. The classification and construction of unipotent representations of $G$ is relatively well-understood, see \cite{LMBM}, \cite{MBMat}. So the functors $L_i\mathcal{R}$ should provide new and useful constructions of unipotent representations of real reductive groups.



\subsection{Acknowledgments}

I am grateful to David Vogan for many helpful discussions.

\section{Restriction of coherent sheaves}\label{sec:coherent}

Choose a Cartan involution $\theta$ of $G$ compatible with $G_{\RR}$ and let $K = G^{\theta}$ (so that $K \cap G_{\RR}$ is a maximal compact subgroup of $G_{\RR}$). Write $\fg$, $\fk$ for the Lie algebras of $G$, $K$ and let $\fp= \fg^{-d\theta}$. 

For any subgroup $H$ of $G$ consider the following subgroups of $G \times G$
\begin{align*}
    H_L &= \{(h,1) \mid h \in H\}\\
    H_R &= \{(1,h) \mid h \in H\}\\
    H_{\Delta} &= \{(h,h) \mid h \in H\}
\end{align*}
Denote the Lie algebras by $\fh_L$, $\fh_R$, $\fh_{\Delta}$, and so on. There are several embeddings which will show up below.

\begin{itemize}
    \item The embedding $\fp \subset \fg$ induces a $K$-equivariant isomorphism
$$(\fg/\fk)^* \xrightarrow{\sim} \fp^*$$
    \item The projection $\fg \to \fp$ induces a $K$-equivariant embedding
    $$\fp^* \hookrightarrow \fg^*$$
    \item The embedding $\fg \hookrightarrow \fg \times \fg$ defined by $x \mapsto \frac{1}{2}(x,-x)$ induces an isomorphism
$$(\fg \times \fg/\fg_{\Delta})^* \xrightarrow{\sim} \fg^*$$
This isomorphism intertwines the $G_{\Delta}$-action on the source with the $G$-action on the target.
\item The embedding $\fg \hookrightarrow \fg \times \fg$ defined by $x \mapsto (x,0)$ induces an isomorphism
$$(\fg \times \fg/(\fg_{\Delta}+\fk_R))^* \xrightarrow{\sim} (\fg/\fk)^*$$
and hence an embedding
$$\iota: (\fg/\fk)^* \hookrightarrow (\fg \times \fg/\fg_{\Delta})^*$$
This embedding intertwines the $K$-action on the source with the $K_{\Delta}$-action on the target. 
\end{itemize}
It is easy to see that the following diagram commutes
\begin{center}
    \begin{tikzcd}
    (\fg/\fk)^* \ar[d,"\sim"] \ar[r,hookrightarrow,"\iota"] & (\fg \times \fg/\fg_{\Delta})^* \ar[d,"\sim"] \\
    \fp^* \ar[r,hookrightarrow] & \fg^*
    \end{tikzcd}
\end{center}
Restriction along $\iota$ defines a right exact functor
$$R: \mathrm{QCoh}^{G_{\Delta}}(\fg \times \fg/\fg_{\Delta})^* \to \mathrm{QCoh}^K(\fg/\fk)^*$$
Write 
$$L_iR: \mathrm{QCoh}^{G_{\Delta}}(\fg \times \fg/\fg_{\Delta})^* \to \mathrm{QCoh}^K(\fg/\fk)^*$$
for the left derived functors. Taking global sections, the category $\mathrm{QCoh}^{G_{\Delta}}(\fg \times \fg/\fg_{\Delta})^*$ (resp. $\mathrm{QCoh}^K(\fg/\fk)^*$) is identified with the category of $G_{\Delta}$-equivariant $S(\fg \times \fg/\fg_{\Delta})$-modules (resp. $K$-equivariant $S(\fg/\fk)$-modules), and $L_iR(M)$ is identified with $\mathrm{Tor}_i^{S(\fg \times \fg/\fg_{\Delta})}(S(\fg/\fk),M)$. These modules can be computed using the Koszul complex $\wedge(\fk_R) \otimes M$
$$0 \to \wedge^d(\fk_R) \otimes M \to \wedge^{d-1}(\fk_R) \otimes M \to ... \to \wedge^1(\fk_R) \otimes M \to M \to 0$$
Here $d= \dim(\fk)$ and the differentials are given by
\begin{align*}
\partial(X_1 \wedge ... \wedge X_n \otimes v) = &\sum_{i=1}^n (-1)^i (X_1 \wedge... \wedge \widehat{X}_i \wedge ... \wedge X_n \otimes X_iv)\\
+ &\sum_{j < k} (-1)^{j+k}([X_j, X_k] \wedge X_1 \wedge ... \wedge \widehat{X}_j \wedge ... \wedge \widehat{X}_k \wedge ... \wedge X_n \otimes v)\end{align*}
The terms in the complex $\wedge(\fk_R) \otimes M$ are $K$-equivariant $S(\fg/\fk)$-modules via $(S(\fg/\fk),K) \simeq (S(\fg \times \fg/\fg_{\Delta}+\fk_R), K_{\Delta})$ and the differentials are $K$-equivariant $S(\fg/\fk)$-module homomorphisms.

\begin{lemma}\label{lem:propsofRcommutative}
Let $M \in \mathrm{Coh}^{G_{\Delta}}(\fg \times \fg/\fg_{\Delta})^*$. The following are true for each $i \geq 0$
\begin{itemize}
    \item[(i)] There is an isomorphism in $\mathrm{QCoh}^K(\fg/\fk)^*$
    $$L_iR(M) \simeq H_i(\wedge(\fk_R) \otimes M)$$
    \item[(ii)] If $i > d=\dim(\fk)$, then $L_iR(M)=0$.
    \item[(iii)] If $M \in \mathrm{Coh}^{G_{\Delta}}(\fg \times \fg/\fg_{\Delta})^*$, then $L_iR(M) \in \mathrm{Coh}^K(\fg/\fk)^*$.
    \item[(iv)] There is an inclusion of sets
    $$\mathrm{Supp}(L_iR(M)) \subseteq \iota^{-1}\mathrm{Supp}(M)$$
\end{itemize}
\end{lemma}

\begin{proof}
By definition, the image of $\iota$ is the subspace $(\fg \times \fg/(\fg_{\Delta}+\fk_R))^* \subset (\fg \times \fg/\fg_{\Delta})^*$, which is defined by the ideal in $S(\fg \times \fg/\fg_{\Delta})$ generated by $\fk_R$. So by standard commutative algebra, $\mathrm{Tor}_i^{S(\fg \times \fg/\fg_{\Delta})}(S(\fg/\fk),M) \simeq H_i(\wedge(\fk_R) \otimes M)$, see \cite[Chapter 16]{matsumura_1987}. This proves (i). (ii) follows immediately. (iii) is a consequence of the fact that $\iota$ is closed, hence proper. (iv) is explained in \cite[Chapter 5.2.5(iii)]{Chriss-Ginzburg}.
\end{proof}

Now let $\cN \subset \fg^*$ be the nilpotent cone. Write
$$\cN_{\theta} := \cN \cap (\fg/\fk)^*, \qquad \cN_{-\Delta} := \{(\xi,-\xi) \mid \xi \in \cN\}$$
Note that $\iota^{-1}(\cN_{-\Delta}) = \cN_{\theta}$. So by Lemma \ref{lem:propsofRcommutative}(iii) and (iv), there are functors
\begin{equation}\label{eq:LiRcommutative}L_iR: \Coh^{G_{\Delta}}(\cN_{-\Delta}) \to \Coh^K(\cN_{\theta})\end{equation}
Taking the alternating sum, we get a group homomorphism
\begin{equation}\label{eq:Rcommutative}R: K^{G_{\Delta}}(\cN_{-\Delta}) \to K^K(\cN_{\theta}), \qquad R[M] = \sum_i (-1)^i [L_iR(M)].\end{equation}
(the sum on the right is finite thanks to Lemma \ref{lem:propsofRcommutative}(iii)). We explained in \cite{MBfunctions} how to compute this homomorphism. The group $\CC^{\times}$ acts by dilation on $(\fg \times \fg)^*$ and this action preserves the subsets $\fg_{\Delta}$, $(\fg/\fk)^*$, $\cN_{\theta}$, and $\cN_{-\Delta}$. There are natural graded versions of (\ref{eq:LiRcommutative}) and (\ref{eq:Rcommutative}) (denoted in the same way) and the following diagram commutes
\begin{center}
    \begin{tikzcd}
    K^{G_{\Delta}}(\cN_{-\Delta}) \ar[r,"R"] & K^K(\cN_{\theta})\\
    K^{G_{\Delta} \times \CC^{\times}}(\cN_{-\Delta}) \ar[r,"R"] \ar[u,twoheadrightarrow] & K^{K\times \CC^{\times}}(\cN_{\theta}) \ar[u,twoheadrightarrow]
    \end{tikzcd}
\end{center}
Let $C$ be the virtual finite-dimensional graded $K$-representation
\begin{equation}\label{eq:Koszul}C = \sum_i (-1)^i \wedge^i(\fk),\end{equation}
(with $\fk$ in degree $1$). Then the homomorphism $R:K^{G_{\Delta} \times \CC^{\times}}(\cN_{-\Delta}) \to K^{K\times \CC^{\times}}(\cN_{\theta})$ is uniquely determined by the following identity of virtual graded $K$-representations
\begin{equation}\label{eq:C}\Gamma(\cN_{\theta},R[M]) = \Gamma(\cN_{-\Delta},[M]) \otimes C\end{equation}
where $\Gamma(\cN_{-\Delta},[M])$ is regarded as a $K$-representation via $K \xrightarrow{\sim} K_{\Delta}$. This is \cite[Corollary 3.0.7]{MBfunctions}. We will also need the following.

\begin{prop}[Corollary 6.0.2, \cite{MBfunctions}]\label{prop:split}
Suppose $G_{\RR}$ is split modulo center. Then 
$$R[\cO_{\cN_{-\Delta}}] = [\cO_{\cN_{\theta}}].$$
\end{prop}

\section{Restriction of representations}\label{sec:functors}

We continue with the notation of Section \ref{sec:coherent}. Consider the categories
\begin{align*}
    M(G_{\RR}) &= \text{category of } (\fg,K)\text{-modules}\\
    M(G) &= \text{category of } (\fg \times \fg, G_{\Delta})\text{-modules}
\end{align*}
Use the subscripts `fl' and `fg' for the full subcategories of finite-length and finitely-generated modules, e.g. $M_{fl}(G_{\RR})$, $M_{fg}(G_{\RR})$ etc. Recall that $M_{fl}(G_{\RR})$  (resp. $M_{fl}(G)$) is equivalent to the category of finite-length admissible representations of $G_{\RR}$ (resp. $G$).

If $X \in M(G)$, then $X/\mathfrak{k}_RX$ has the structure of a $(\fg,K)$-module---$\fg$ acts via $\fg \simeq \fg_L \subset \fg \times \fg$ and $K$ acts via $K \simeq K_{\Delta} \subset G_{\Delta}$. This defines a right-exact functor
$$\mathcal{R}: M(G) \to M(G_{\RR}), \qquad \mathcal{R}(X) = X/\mathfrak{k}_R X$$
Since the category $M(G)$ has enough projectives, see \cite[Cor 2.37]{KnappVogan1995}, we can form the left-derived functors 
$$L_i\mathcal{R}: M(G) \to M(G_{\RR}), \qquad i \geq 0.$$
As a vector space, $L_i\mathcal{R}(X)$ is just the Lie algebra homology $H_i^{\mathfrak{k}_2}(X,\CC)$. 

To compute $L_i\mathcal{R}(X)$, we form the Koszul complex $\wedge(\fk_R) \otimes X$
\begin{equation}\label{eq:Koszul1}0 \to \wedge^d(\fk_R) \otimes X \to \wedge^{d-1}(\fk_R) \otimes X \to ... \to \wedge^1(\fk_R) \otimes X \to X \to 0\end{equation}
Here $d=\dim(\fk)$. The differentials are given by
\begin{align*}
\partial(X_1 \wedge ... \wedge X_n \otimes v) = &\sum_{i=1}^n (-1)^i (X_1 \wedge... \wedge \widehat{X}_i \wedge ... \wedge X_n \otimes X_iv)\\
+ &\sum_{j < k} (-1)^{j+k}([X_j, X_k] \wedge X_1 \wedge ... \wedge \widehat{X}_j \wedge ... \wedge \widehat{X}_k \wedge ... \wedge X_n \otimes v)\end{align*}
Note that the terms in the complex $\wedge(\fk_R) \otimes X$ are $(\fg,K)$-modules (again, via $(\fg,K) \simeq (\fg_L, K_{\Delta})$), and the differentials are $(\fg,K)$-module homomorphisms. We begin by recording some elementary properties of the modules $L_i\mathcal{R}(X)$. Choose a Cartan subalgebra $\fh \subset \fg$ and let $W$ denote the Weyl group of $G$.

\begin{lemma}\label{lem:propsofR}
Let $X \in M(G)$. Then the following are true for each $i \geq 0$
\begin{itemize}
    \item[(i)] There is a $(\fg,K)$-module isomorphism
$$L_i\mathcal{R}(X) \simeq H_i(\wedge(\fk_R)\otimes X).$$
    \item[(ii)] $L_i\mathcal{R}(X)=0$ for $i > d = \dim(\fk)$.
    \item[(iii)] If $X$ has finite support over $\Spec(\zeta(\fg)) \times \Spec(\zeta(\fg))=\Spec(\zeta(\fg \times \fg))$, then $L_i\mathcal{R}(X)$ has finite support over $\Spec(\zeta(\fg))$.
\item[(iv)] If $X$ has infinitesimal character $(\lambda_L,\lambda_R) \in \fh^*/W \times \fh^*/W$, then $L_i\mathcal{R}(X)$ has infinitesimal character $\lambda_L$.
\end{itemize}
\end{lemma}

\begin{proof}
(i) is a special case of \cite[Thm 2.122]{KnappVogan1995}. (ii) follows from (i). By (i), we have that $L_i\mathcal{R}(X)$ is a $\fz(U(\fg))$-module subquotient of $\wedge^i(\fk_R) \otimes X$. Hence
\begin{align*}
\mathrm{Supp}_{\mathrm{Spec}(\fz(U(\fg)))}(L_i(\mathcal{R}(X)) &\subseteq \mathrm{Supp}_{\mathrm{Spec}(\fz(U(\fg)))}(\wedge^i(\fk_R) \otimes X)\\
&= \mathrm{Supp}_{\mathrm{Spec}(\fz(U(\fg)))}(X)\\
&= \mathrm{Spec}(\fz(U(\fg))) \cap \mathrm{Supp}_{\mathrm{Spec}(\fz(U(\fg))) \times \mathrm{Spec}(\fz(U(\fg)))}(X)\end{align*}
(iii) follows at once. If $X$ has infinitesimal character $(\lambda_L,\lambda_R)$, then each $\wedge^i(\fk_R) \otimes X$ has infinitesimal character $\lambda_L$ as a $U(\fg_L)$-module. So (iv) also follows from (i).
\end{proof}

To go further, we will need to bring \emph{good filtrations} into the picture. For details and proofs, we refer the reader to \cite[Section 2]{Vogan1991}. Suppose $X \in M(G_{\RR})$. A compatible filtration is an increasing filtration by $K$-invariant subspaces $\{X_n \mid n \in \ZZ\}$ such that $U_m(\fg)X_n \subseteq X_{m+n}$. If $X$ is equipped with a compatible filtration, then $\gr(X)$ has the structure of a $K$-equivariant $S(\fg/\fk)$-module (equivalently, of an object in $\mathrm{QCoh}^K(\fg/\fk)^*$). The filtration is said to be \emph{good} if this module is finitely-generated (equivalently, if the associated sheaf in $\mathrm{QCoh}^K(\fg/\fk)^*$ is coherent). Note that $X$ admits a good filtration if and only if $X \in M_{fg}(G_{\RR})$. Good filtrations are not unique, but the class $[\gr(X)]$ in $K^K(\fg/\fk)^*$ is well-defined. The \emph{associated variety} of $X$ is the $K$-invariant subset $\mathcal{V}(X) = \mathrm{Supp}(\gr(X)) \subset (\fg/\fk)^*$ (since support is additive on short exact sequences, $\mathcal{V}(X)$ is well-defined). If $X \in M_{fl}(G_{\RR})$, then $\mathcal{V}(X)$ is contained in $\cN_{\theta}$. Thus, we obtain a group homomorphism
$$\gr: KM_{fl}(G_{\RR}) \to K^K(\cN_{\theta}).$$
Similarly (replacing $G$ with $G \times G$ and $K$ with $G_{\Delta}$) we get a group homomorphism
$$\gr: KM_{fl}(G) \to K^{G_{\Delta}}(\cN_{-\Delta}).$$
Our goal is to prove

\begin{prop}\label{prop:propsofR}
The following are true:
\begin{itemize}
    \item[(i)] For $i \geq 0$, the functors $L_i\mathcal{R}$ preserve finite generation and finite length, i.e. they restrict to functors
    $$L_i\mathcal{R}: M_{fg}(G) \to M_{fg}(G_{\RR}), \qquad L_i\mathcal{R}: M_{fl}(G) \to M_{fl}(G_{\RR}).$$
\end{itemize}
Let $X \in M_{fg}(G)$ and choose a good filtration $\{X_{n}\}$ of $X$. Then
\begin{itemize}
    \item[(ii)] The filtration on $X$ induces natural compatible filtrations on $L_i\mathcal{R}(X)$ for $i \geq 0$.
    \item[(iii)] The filtrations in (ii) are good.
    \item[(iv)]  $\mathcal{V}(L_i\mathcal{R}(X)) \subseteq \mathrm{Supp}_{(\fg/\fk)^*}(L_iR(\gr(X)))$ for $i \geq 0$.
    \item[(v)] If $X \in M_{fl}(G)$, then $\gr (L_i\mathcal{R}(X))$ and $L_iR(\gr(X))$ are classes in $K^K(\cN_{\theta})$ and there is an identity
    $$\sum_{i=0}^d (-1)^i [\gr(L_i\mathcal{R}(X))] = \sum_{i=0}^d (-1)^i [L_iR(\gr(X))]$$
    %
\end{itemize}
\end{prop}

The proof will come after a lemma.

\begin{lemma}\label{lem:spectral}
Let $R$ be a filtered ring with an incrasing filtration $R=\bigcup R^m$ such that $S=\gr(R)$ is commutative and Noetherian. Suppose $(A,d)$ is a bounded complex of $R$-modules
$$0 \to {}^0A \overset{d}{\to} {}^1A \overset{d}{\to} ... \overset{d}{\to} {}^dA \to 0$$
Choose an increasing filtration ${}^iA = \bigcup_p {}^{i}A^p$ on each ${}^iA$ such that
\begin{itemize}
    \item $R^m({}^iA^p) \subseteq {}^iA^{m+p}$.
    \item $d({}^iA^p) \subseteq {}^{i+1}A^{p}$.
\end{itemize}
Form the spectral sequence $\{E_r \mid r \geq 0\}$ associated to the filtered complex $A$. Write
$${}^nE_r := \bigoplus_{p+q=n} E_r^{p,q}$$
so that $E_r$ is a complex
$$...\to {}^0E_r \overset{d_r}{\to} {}^1E_r \overset{d_r}{\to} ... \overset{d_r}{\to} {}^dE_r \to ... $$
(with degree-$r$ differentials). Then
\begin{itemize}
    \item[(i)] Each ${}^iE_r$ and ${}^iE_{\infty}$ is naturally a graded $S$-module. The graded $S$-module structure on ${}^iE_{\infty}$ coincides with the graded $S$-module structure on $\gr(H_i(A))$ under the natural identification $E_{\infty} \simeq \gr(H(A))$.
    \item[(ii)] ${}^iE_{\infty} = {}^iE_r = 0$ for $i<0$ or $i > d$.
    \item[(iii)] ${}^iE_{\infty}$ and ${}^iE_{r+1}$ are $S$-module subquotients of ${}^iE_r$ for all $r$ and $i$.
\end{itemize}
Suppose, in addition, that ${}^iE_1$ are \emph{finitely-generated} $S$-modules for all $i$. Then
\begin{itemize}
    \item[(iv)] ${}^iE_r$ and ${}^iE_{\infty}$ are finitely-generated $S$-modules for all $r$ and $i$.
    \item[(v)] There is an inclusion $\mathrm{Supp}_{\mathrm{Spec}(S)}({}^iE_{\infty}) \subseteq \mathrm{Supp}_{\mathrm{Spec}(S)}({}^iE_r)$ for all $r$ and $i$.
    \item[(vi)] There is an identity in the Grothendieck group of finitely-generated $S$-modules
    $$\sum_{i=0}^d (-1)^i ({}^iE_r) = \sum_{i=0}^d (-1)^i ({}^iE_{\infty})$$
    for all $r$.
\end{itemize}
\end{lemma}

\begin{proof}
As usual, let
\begin{align*}
    A^{p,q} &= {}^{p+q}A \cap A^p\\
    Z_r^{p,q} &= \{z \in A^{p,q} \mid dz \in A^{p-r,q+r+1}\}\\
    E_r^{p,q} &= Z_r^{p,q}/(Z_{r-1}^{p-1,q+1} + dZ_{r-1}^{p+r-1,q-r})\\
    Z_{\infty}^{p,q} &= A^{p,q} \cap \ker{(d)}\\
    B^{p,q}_{\infty} &= A^{p,q} \cap \mathrm{im}(d)\\
    E_{\infty}^{p,q} &= Z_{\infty}^{p,q}/(Z_{\infty}^{p-1,q+1}+B_{\infty}^{p,q})
\end{align*}
Note that 
\begin{enumerate}
    \item $R^m(A^{p,q}) \subseteq A^{p+m,q-m}$.
    \item $R^m(Z_r^{p,q}) \subseteq Z_r^{p+m,q-m}$.
    \item $R^m(Z_{\infty}^{p,q}) \subseteq Z_{\infty}^{p+m,q-m}$.
    \item $R^m(B_{\infty}^{p,q}) \subseteq B_{\infty}^{p+m,q-m}$.
\end{enumerate}
It follows that
$$R^m(Z_{r-1}^{p-1,q+1} + dZ_{r-1}^{p+r-1,q-r}), R^{m-1}(Z_r^{p,q}) \subseteq Z_{r-1}^{p+m-1,q-m+1} + dZ_{r-1}^{p+m+r-1,q-m-r}$$
and
$$R^m(Z_{\infty}^{p-1,q+1} + B_{\infty}^{p,q}), R^{m-1}(Z_{\infty}^{p,q}) \subseteq Z_{\infty}^{p+m-1,q-m+1} + B_{\infty}^{p+m,q-m}$$
So $R^m$ carries $E_r^{p,q}$ to $E_r^{p+m,q-m}$ and $E_{\infty}^{p,q}$ to $E_{\infty}^{p+m,q-m}$. This proves (i).

It is well-known that $^{i}E_{r+1} = H_i({}^{\bullet}E_r)$. (ii) follows. (iii) is immediate from the definitions of ${}^iE_r$ and ${}^iE_{\infty}$.

Suppose $X$ and $Y$ are $S$-modules and $X$ is a subquotient of $Y$. Since $S$ is Noetherian
$$Y \text{ finitely-generated} \implies X \text{ finitely-generated}$$
Since support is additive on short exact sequences
$$\mathrm{Supp}_{\mathrm{Spec}(S)}(X) \subseteq \mathrm{Supp}_{\mathrm{Spec}(S)}(Y)$$
Now (iv) and (v) follow immediately from (iii). (vi) follows by induction on $r$, see \cite[D.30]{KnappVogan1995}.
\end{proof}

\begin{proof}[Proof of Proposition \ref{prop:propsofR}]
Let $X \in M_{fg}(G)$ and choose a good filtration $\{X_m\}$ of $X$. Each term in the Koszul complex \ref{eq:Koszul1} acquires an increasing filtration (with $\fk_R$ in degree $1$). These filtrations induce natural compatible filtrations on the homology groups $H_i(\wedge(\fk_R) \otimes X)$ and hence, by Lemma \ref{lem:propsofR}(i), on the modules $L_i\mathcal{R}(X)$. This proves (ii). 

By Lemma \ref{lem:propsofRcommutative}(i)
$${}^iE_1 \simeq H_i(\wedge(\fk_R) \otimes \gr(X)) \simeq L_iR(\gr(X)), \qquad \forall i$$
Since $\gr(X) \in \Coh^{G_{\Delta}}(\fg \times \fg/\fg_{\Delta})^*$, we have $L_iR(\gr(X)) \in \Coh^K(\fg/\fk)^*$ by Lemma \ref{lem:propsofRcommutative}(ii). On the other hand
$${}^iE_{\infty} \simeq \gr(L_i\mathcal{R}(X)), \qquad \forall i$$
So by Lemma \ref{lem:spectral}(iv), $\gr(L_i\mathcal{R}(X)) \in \Coh^K(\fg/\fk)^*$, i.e. the filtration on $L_i\mathcal{R}(X)$ are good. This proves (iii). (iv) is an immediate consequence of Lemma \ref{lem:propsofR}(v). And (v) is an immediate consequence of Lemma \ref{lem:propsofR}(vi). 

If a $(\fg,K)$-module admits a good filtration, it must be finitely-generated. So $L_i\mathcal{R}(X) \in M_{fg}(G_{\RR})$ by (iii). 
A $(\fg,K)$-module has finite length if and only if it is finitely-generated and finitely supported over $\mathrm{Spec}(\fz(\fg))$. So if $X \in M_{fl}(G)$, then $L_i\mathcal{R}(X) \in M_{fl}(G_{\RR})$ by Lemma \ref{lem:propsofR}(ii).
This proves (i). 
\end{proof}

Consider the group homomorphism
$$\mathcal{R}: KM_{fl}(G) \to KM_{fl}(G_{\RR}), \qquad \mathcal{R}[X] = \sum_{i=0}^{\infty} [L_i\mathcal{R}X]$$
This is well-defined by Lemma \ref{lem:propsofR}(ii) and Proposition \ref{prop:propsofR}(i). The following is an immediate consequence of Proposition \ref{prop:propsofR}(iv).

\begin{cor}\label{cor:Rgr}
The following diagram commutes
\begin{center}
    \begin{tikzcd}
    KM_{fl}(G) \ar[d,"\gr"] \ar[r,"\mathcal{R}"] & KM_{fl}(G_{\RR}) \ar[d,"\gr"] \\
    K^{G_{\Delta}}(\cN_{-\Delta}) \ar[r,"R"] & K^K(\cN_{\theta})
    \end{tikzcd}
\end{center}
\end{cor}

The following is a consequence of Corollary \ref{cor:Rgr} and Equation (\ref{eq:C}). 

\begin{cor}\label{cor:Ktypes}
Let $X \in KM_{fl}(G)$, and choose an (arbitrary) good filtration. Then as representations of $K$
$$\mathcal{R}[X] =_K \gr[X] \otimes C$$
where $C$ is the signed graded Koszul class (cf. \ref{eq:Koszul}). 
\end{cor}

\section{Case of split groups}\label{sec:split}

Now assume that $G_{\RR}$ is split. Fix a split maximal torus $H \subset G$ and a Borel subgroup $B=HN \subset G$. Write $H_{\RR} \subset G_{\RR}$ and $B_{\RR} \subset G_{\RR}$ for the groups of real points. Let 
$$H= TA$$
be the Cartan decomposition of $H$. Here, $T$ is a product of circle groups and $A$ is a complex vector group. Let $T_{\RR} = T \cap H_{\RR}$ and $A_{\RR} = A \cap H_{\RR}$. Then $T_{\RR}$ is a product of cyclic groups of order 2, $A_{\RR}$ is a real vector group, and
$$H_{\RR} = T_{\RR}A_{\RR}$$
is the Cartan decomposition of $H_{\RR}$. There are bijections
\begin{equation}\label{eq:characters}
    \widehat{H} \simeq \widehat{T} \times \widehat{A} \simeq X^*(H) \times \fh^*, \qquad 
    \widehat{H}_{\RR} \simeq \widehat{T}_{\RR} \times \widehat{A}_{\RR} \simeq X^*(H)/2X^*(H) \times \fh^*
\end{equation}
IF $(\lambda,\nu) \in X^*(H) \times \fh^*$ (resp. $(\overline{\lambda},\nu) \in X^*(H)/2X^*(H) \times \fh^*$), write $\CC(\lambda,\nu)$ (resp. $\CC(\overline{\lambda},\nu)$) for the corresponding character of $H$ (resp. $H_{\RR}$). Note that the character $\CC(\lambda,\nu)$ of $H$ has infinitesimal character $\frac{1}{2}(\nu + \lambda, \nu - \lambda) \in \fh^*\times \fh^*$ and the character $\CC(\overline{\lambda},\nu)$ has infinitesimal character $\nu \in \fh^*$. Denote the principal series of $G$ and $G_{\RR}$ by
$$I_{\CC}(\lambda, \nu) := \Ind^G_B \CC(\lambda,\nu), \qquad I_{\RR}(\overline{\lambda},\nu) := \Ind^{G_{\RR}}_{B_{\RR}} \CC(\overline{\lambda},\nu) \qquad \text{(normalized induction)}.$$
Since the induction is normalized, $I_{\CC}(\lambda,\nu)$ has infinitesimal character $\frac{1}{2}(\nu + \lambda, \nu - \lambda) \in \fh^*/W \times \fh^*/W$ and $I_{\RR}(\overline{\lambda},\nu)$ has infinitesimal character $\nu \in \fh^*/W$.

Write $\Lambda \subset X^*(H)$ for the lattice of weights of algebraic representations of $G$ and $\Lambda_{\CC} \subset \widehat{H}$ (resp. $\Lambda_{\RR} \subset \widehat{H}_{\RR}$) for the lattice of weights of finite-dimensional representations of $G$ (resp. $G_{\RR}$). Then, under the bijections (\ref{eq:characters})
$$\Lambda_{\CC} \simeq \{(\epsilon,\delta) \in X^*(H) \times \fh^* \mid  \frac{1}{2}(\delta \pm \epsilon) \in \Lambda\}, \qquad \Lambda_{\RR} \simeq \{(\overline{\epsilon},\epsilon) \in X^*(H)/2X^*(H) \times \fh^* \mid \epsilon \in \Lambda\}.$$
For $(\epsilon,\delta) \in \Lambda_{\CC}$, the Langlands subquotient $F_{\CC}(\epsilon,\delta)$ of $I_{\CC}(\epsilon,\delta+2\rho)$ is an irreducible finite-dimensional representation of $G$. As a representation of $\fg \times \fg$, it is isomorphic to $F(\frac{1}{2}(\delta+\epsilon)) \otimes F(\frac{1}{2}(\delta-\epsilon))$, where $F(\epsilon)$ is the irreducible algebraic representation of $G$ with extremal weight $\epsilon$. For $(\overline{\epsilon},\epsilon) \in \Lambda_{\RR}$, the Langlands subquotient $F_{\RR}(\overline{\epsilon},\epsilon)$ of $I_{\RR}(\overline{\epsilon},\epsilon+\rho)$ is an irreducible finite-dimensional representation of $G_{\RR}$. It is isomorphic to the restriction of the algebraic representation $F(\epsilon)$ of $G$ to $G_{\RR}$. 

Recall that the classes $[I_{\CC}(\lambda,\nu)]$ span the Grothendieck group $KM_{fl}(G)$. So to compute the homomorphism $\mathcal{R}: KM_{fl}(G) \to KM_{fl}(G_{\RR})$ it suffices to compute the classes $\mathcal{R} [I_{\CC}(\lambda,\nu)]$ for each $(\lambda,\nu)$. In fact, we will prove the following.

\begin{theorem}\label{thm:principalseries}
For each $(\lambda, \nu) \in X^*(H) \times \fh^*$ there is an identity in $KM_{fl}(G_{\RR})$
$$\mathcal{R} [I_{\CC}(\lambda,\nu)] = [I_{\RR}(\overline{\lambda},\frac{1}{2}(\lambda+\nu))].$$
\end{theorem}

The proof will come after several lemmas. The first step is to prove Theorem \ref{thm:principalseries} in the special case $\lambda=\nu=0$.

\begin{lemma}\label{lem:grI00}
The following are true:
\begin{itemize}
    \item[(i)] $\gr[I_{\RR}(\overline{0},0)] = [\cO_{\cN_{\theta}}]$ in $K^K(\cN_{\theta})$.
    \item[(ii)] $\gr[I_{\CC}(0,0)] = [\cO_{\cN_{-\Delta}}]$ in $K^{G_{\Delta}}(\cN_{-\Delta})$.
\end{itemize}
\end{lemma}

\begin{proof}
Kostant shows in \cite{kostant1969} that $[I_{\RR}(\overline{0},0)] =_K \CC[\cN_{\theta}]$ as representations of $K$. Since $K$ is reductive, $\gr[I_{\RR}(\overline{0},0)] =_K [I_{\RR}(\overline{0},0)]$. Hence, $\gr[I_{\RR}(\overline{0},0)] =_K \CC[\cN_{\theta}]$. By \cite[Corollary 10.9]{AdamsVoganAV}, every class in $K^K(\cN_{\theta})$ is completely determined by its restriction to $K$. So in fact there is an identity $\gr[I_{\RR}(\overline{0},0)] = [\cO_{\cN_{\theta}}]$ in $K^K(\cN_{\theta})$. Kostant's result holds for the spherical principal series of an arbitrary quasi-split group. In particular, it holds for the complex group $G$. Thus, we have an identity $\gr[I_{\CC}(0,0)] = [\cO_{\cN_{-\Delta}}]$ in $K^{G_{\Delta}}(\cN_{-\Delta})$ by an identical argument.
\end{proof}

\begin{lemma}\label{lem:RSPS}
There is an identity in $KM_{fl}(G_{\RR})$
$$\mathcal{R}[I_{\CC}(0,0)] = [I_{\RR}(0,0)].$$
\end{lemma}

\begin{proof}
By Corollary \ref{cor:Rgr}
\begin{equation}\label{eq:gr1}\gr \mathcal{R}[I_{\CC}(0,0)] = R \gr[I_{\CC}(0,0)]\end{equation}
By Lemma \ref{lem:grI00} and Proposition \ref{prop:split}
$$R\gr[I_{\CC}(0,0)] = \gr[I_{\RR}(\overline{0},0)].$$
Thus,
\begin{equation}\label{eq:gr5}\gr \mathcal{R}[I_{\CC}(0,0)] = \gr[I_{\RR}(0,0)]\end{equation}
Write $KM_0(G_{\RR})$ for the Grothendieck group of finite-length $(\fg,K)$-modules with infinitesimal character $0$ and $KM_{t,\RR}(G_{\RR})$ for the Grothendieck group of finite-length tempered $(\fg,K)$-modules with real infinitesimal character. Note that
$$KM_0(G_{\RR}) \subset KM_{t,\RR}(G_{\RR})$$
By Lemma \ref{lem:propsofR}(iii), $\mathcal{R}[I_{\CC}(0,0)] \in KM_0(G_{\RR})$ and hence $\mathcal{R}[I_{\CC}(0,0)] \in KM_{t,\RR}(G_{\RR})$. By \cite[Corollary 7.4]{AdamsVoganAV}, the restriction of the associated graded map $\gr: KM(G_{\RR}) \to K^K(\cN_{\theta})$ to the subspace $KM_{t,\RR}(G_{\RR})$ is injective. So (\ref{eq:gr5}) implies that $\mathcal{R}[I_{\CC}(0,0)] = [I_{\RR}(0,0)]$ in $KM(G_{\RR})$, as asserted.
\end{proof}

The next step is to show that the homomorphism $\mathcal{R}: KM_{fl}(G) \to KM_{fl}(G_{\RR})$ is compatible with coherent families. For a basic discussion of coherent families, we refer the reader to \cite[Chapter 7.2]{Vogan1981}.

\begin{lemma}\label{lem:Rcoherent}
For $(\lambda,\nu) \in X^*(H) \times \fh^*$, the map
$$\Theta_{(\lambda,\nu)}: \Lambda_{\RR} \to KM_{fl}(G_{\RR}), \qquad \Theta_{(\lambda,\nu)}(\overline{\gamma},\gamma) = \mathcal{R}[I_{\CC}(\lambda+\gamma,\nu+\gamma)]$$
is a coherent family.
\end{lemma}

\begin{proof}
Define
$$\Theta_{\CC}: \Lambda_{\CC} \to KM_{fl}(G), \qquad \Theta_{\CC}(\gamma,\gamma') = [I_{\CC}(\lambda+\gamma,\nu+\gamma')]$$
Note that $\Theta_{\CC}$ is a coherent family in $KM_{fl}(G)$, see \cite[Example 7.2.11]{Vogan1981}. Hence
$$[I_{\CC}(\lambda + \gamma,\nu+\gamma')] \otimes [F_{\CC}(\epsilon,\epsilon')] = \sum_{(\mu,\mu') \in \Delta(F_{\CC}(\epsilon,\epsilon')} [I_{\CC}(\lambda+\gamma+\mu,\nu+\gamma'+\mu')] $$
In particular (setting $\gamma'=\gamma$ and $\epsilon'=\epsilon$)
\begin{equation}\label{eq:cohfamilies}[I_{\CC}(\lambda + \gamma,\nu+\gamma)] \otimes [F_{\CC}(\epsilon,\epsilon)] = \sum_{\mu \in \Delta(\epsilon)} [I_{\CC}(\lambda+\gamma+\mu,\nu+\gamma+\mu)]\end{equation}
Since $F_{\CC}(\epsilon,\epsilon)$ is a trivial $\fk_R$-module
\begin{equation}\label{eq:RF}\mathcal{R}[X \otimes F_{\CC}(\epsilon,\epsilon)] = \mathcal{R}[X] \otimes [F_{\RR}(\overline{\epsilon},\epsilon)]\end{equation}
Applying $\mathcal{R}$ to both sides of (\ref{eq:cohfamilies}) and using (\ref{eq:RF}) to simplify, we get
$$\mathcal{R}[I_{\CC}(\lambda+\gamma,\nu+\gamma)] \otimes [F_{\RR}(\overline{\epsilon},\epsilon)] = \sum_{\mu \in F(\epsilon)} \mathcal{R}[I_{\CC}(\lambda+\gamma+\mu,\nu+\gamma+\mu)]$$
This proves the assertion.
\end{proof}

\begin{proof}[Proof of Theorem \ref{thm:principalseries}]
By Lemma \ref{lem:Rcoherent} (applied in the case when $\lambda=\nu=0$), the map
$$\Theta_{(0,0)}:\Lambda_{\RR} \to KM_{fl}(G_{\RR}), \qquad \Theta_{(0,0)}(\overline{\gamma},\gamma) = \mathcal{R}[I_{\CC}(\gamma,\gamma)]$$
is a coherent family. And by Lemma \ref{lem:RSPS}, $\Theta_{(0,0)}(\overline{0},0) = [I_{\RR}(0,0)]$. The map
$$\Lambda_{\RR} \to KM_{fl}(G_{\RR}), \qquad (\overline{\gamma},\gamma) \mapsto [I_{\RR}(\gamma,\gamma)]$$
is also a coherent family (cf. \cite[Example 7.2.11]{Vogan1981}) with the same value at $(\overline{0},0)$. So by the uniqueness of coherent families (\cite[Corollary 7.2.27]{Vogan1981}), we have
$$\mathcal{R}[I_{\CC}(\gamma,\gamma)] = [I_{\RR}(\overline{\gamma},\gamma)], \qquad \forall \gamma \in \Lambda$$
Taking $\gr$ and applying Corollary \ref{cor:Rgr}, we deduce
\begin{equation}\label{eq:1}R \gr[I_{\CC}(\gamma,\gamma)] = \gr[I_{\RR}(\overline{\gamma},\gamma)], \qquad \forall \gamma \in \Lambda.\end{equation}
Recall that the associated graded of $I_{\CC}(\lambda,\nu)$ (or $I_{\RR}(\overline{\lambda},\nu)$) is independent of the continuous parameter $\nu$, see \cite[p. 54, (3)]{AdamsVoganAV}. So (\ref{eq:1}) implies
$$R \gr[I_{\CC}(\lambda+\gamma,\nu+\gamma)] = \gr [I_{\RR}(\overline{\lambda+\gamma},\frac{1}{2}(\lambda+\nu)+\gamma)]$$
Applying Corollary \ref{cor:Rgr} we obtain
\begin{equation}\label{eq:2}\gr \mathcal{R}[I_{\CC}(\lambda+\gamma,\nu+\gamma)] = \gr [I_{\RR}(\overline{\lambda+\gamma},\frac{1}{2}(\lambda+\nu)+\gamma)]\end{equation}
By Lemma \ref{lem:Rcoherent}, $\Theta_{(\lambda,\nu)}$ is a coherent family. The map $(\gamma,\overline{\gamma}) \mapsto [I_{\RR}(\overline{\lambda+\gamma},\nu+\gamma)]$ is also a coherent family (see \cite[Example 7.2.11]{Vogan1981}). Thus we have constructed two coherent families with the same associated gradeds (\ref{eq:2}) and infinitesimal characters (Lemma \ref{lem:propsofR}(iii)). So by \cite[Lemma 8.2]{VoganIC2}, they are the same. In particular,
$$\mathcal{R}[I_{\CC}(\lambda,\nu)] = [I_{\RR}(\overline{\lambda},\frac{1}{2}(\lambda+\nu))],$$
as asserted.
\end{proof}

\section{Unipotent representations}\label{sec:unipotent}

A \emph{complex nilpotent cover} for $G_{\RR}$ is a finite connected $G$-equivariant cover $\widetilde{\OO}$ of a nilpotent co-adjoint $G$-orbit $\OO \subset \fg^*$. Write $\mathsf{Cov}_n(G_{\RR})$ for the set of isomorphism classes of complex nilpotent covers. In \cite{LMBM}, we associate a primitive ideal $I(\widetilde{\OO}) \subset U(\fg)$ to each nilpotent cover $\widetilde{\OO} \in \mathsf{Cov}_n(G_{\RR})$. We call $I(\widetilde{\OO})$ the \emph{unipotent ideal} attached to $\widetilde{\OO}$. Below, we record some basic properties of unipotent ideals. Recall that the \emph{associated variety} of a two-sided ideal $I \subset U(\fg)$ is the $G$-invariant subset $V(I) \subset \fg^*$ defined by the ideal $\gr(I) \subset S(\fg) \simeq \CC[\fg^*]$.

\begin{prop}[Proposition 6.1.2, \cite{LMBM} and Theorem 5.0.1,\cite{MBMat}]\label{prop:propsofunipotent}
For each $\widetilde{\OO} \in \mathsf{Cov}_n(G_{\RR})$, the ideal $I(\widetilde{\OO}) \subset U(\fg)$ has the following properties:
\begin{itemize}
    \item[(i)] $I(\widetilde{\OO})$ is completely prime.
    \item[(ii)] $I(\widetilde{\OO})$ is maximal (among two-sided ideals in $U(\fg)$).
    \item[(iii)] $V(I(\widetilde{\OO})) = \overline{\OO}$.
\end{itemize}
\end{prop}

\begin{definition}\label{def:unipotent}
Let $\widetilde{\OO} \in \mathsf{Cov}_n(G_{\RR})$. Write $M_{\widetilde{\OO}}(G_{\RR})$ for the full subcategory of $M_{fl}(G_{\RR})$ consisting of $(\fg,K)$-modules $X$ annihilated by the ideal $I(\widetilde{\OO})$. A \emph{unipotent representation} of $G_{\RR}$ attached to $\widetilde{\OO}$ is an irreducible object in $M_{\widetilde{\OO}}(G_{\RR})$.
\end{definition}

\begin{rmk}
If $\OO^{\vee} \subset \fg^{\vee}$ is a nilpotent adjoint orbit for the Langlands dual group $G^{\vee}$, there is a notion of a \emph{special unipotent representation} attached to $\OO^{\vee}$, defined by Barbasch and Vogan in \cite{BarbaschVogan1985}. In \cite{LMBM}, we attach to $\OO^{\vee}$ a complex nilpotent cover $d(\OO^{\vee}) \in \mathsf{Cov}_n(G_{\RR})$ such that the special unipotent representations attached to $\OO^{\vee}$ coincide with the unipotent representations attached to $d(\OO^{\vee})$. Thus, Definition \ref{def:unipotent} generalizes the notion of `special unipotent'.
\end{rmk}

If $X \in M_{fl}(G_{\RR})$, the associated variety $\mathcal{V}(X)$ of $X$ is the $G$-saturation of the associated variety $V(\mathrm{Ann}(X))$ of its annihilator. 

\begin{lemma}\label{lem:unipotent}
Let $\widetilde{\OO} \in \mathsf{Cov}_n(G_{\RR})$ and let $X \in M_{fl}(G_{\RR})$. Then $X \in M_{\widetilde{\OO}}(G_{\RR})$ if and only if:
\begin{itemize}
    \item[(i)] $X$ has infinitesimal character $\gamma(\widetilde{\OO})$, and
    \item[(ii)] $\mathcal{V}(X) \subseteq \overline{\OO}$.
\end{itemize}
\end{lemma}

\begin{proof}
The `only if' direction is trivial. We proceed to the `if' direction. Suppose $X$ satisfies conditions (i) and (ii). Let $Y$ be an irreducible composition factor of $X$ and let $I\subset U(\fg)$ be the annihilator of $Y$. It suffices to show that $I=I(\widetilde{\OO})$.

Property (i) implies that $I$ is a primitive ideal with infinitesimal character $\gamma(\widetilde{\OO})$. Property (ii) implies that $V(I) = G\mathcal{V}(Y) \subset G\mathcal{V}(X) \subseteq \overline{\OO}$.

By an old result of Duflo (\cite{Duflo1977}), there is a unique maximal ideal of each infinitesimal character, and $I(\widetilde{\OO})$ is maximal by Proposition \ref{prop:propsofunipotent}(ii). So there must be an inclusion $I \subseteq I(\widetilde{\OO})$. If this inclusion is proper, then $V(I) \supsetneq V(I(\widetilde{\OO})) = \overline{\OO}$ (see \cite[Corollary 3.6]{BorhoKraft}). This is a contradiction. So indeed, $I = I(\widetilde{\OO})$, as asserted.
\end{proof}

Note that $\mathsf{Cov}_n(G) \simeq \mathsf{Cov}_n(G_{\RR}) \times \mathsf{Cov}_n(G_{\RR})$. If $\widetilde{\OO}_L \times \widetilde{\OO}_R \in \mathsf{Cov}_n(G)$, then 
$$\OO_L \cap \OO_R \cap (\fg \times \fg/\fg_{\Delta})^* \neq \emptyset \implies \OO_L = \OO_R$$
So $M_{\widetilde{\OO}_L \times \widetilde{\OO}_R}(G)=0$ unless $\OO_L=\OO_R$. 

\begin{cor}
Let $\widetilde{\OO}_L \times \widetilde{\OO}_R \in \mathsf{Cov}_n(G)$. Then for each $0 \leq i \leq d$, the functor $L_i\mathcal{R}$ takes $M_{\widetilde{\OO}_L \times \widetilde{\OO}_R}(G)$ to $M_{\widetilde{\OO}_L}(G_{\RR})$.  
\end{cor}

\begin{proof}
Suppose $X \in M_{\widetilde{\OO}_L \times \widetilde{\OO}_R}(G)$. Then by Lemma \ref{lem:unipotent}, $(\gamma_L,\gamma_R) = (\gamma(\widetilde{\OO}_L), \gamma(\widetilde{\OO}_R))$ and $\mathcal{V}(X) \subseteq \overline{\OO}_L \times \overline{\OO}_L$. By Proposition \ref{prop:propsofR}(v), there is an inclusion
$$\mathcal{V}(L_i\mathcal{R}(X)) \subseteq \mathrm{Supp}(L_iR(\gr(X)))$$
and by Lemma \ref{lem:propsofRcommutative}(iv), there is an inclusion
$$\mathrm{Supp}(L_iR(\gr(X))) \subseteq \mathcal{V}(X) \cap \fg_L^* \subseteq \overline{\OO}_L.$$
Hence
$$\mathcal{V}(L_i\mathcal{R}(X)) \subseteq \overline{\OO}_L.$$
On the other hand, $L_i\mathcal{R}(X)$ has infinitesimal character $\gamma(\widetilde{\OO}_L)$ by Lemma \ref{lem:propsofR}(iii). Hence, $L_i\mathcal{R}(X) \in M_{\widetilde{\OO}_L}(G_{\RR})$ by a second application of Lemma \ref{lem:unipotent}. 
\end{proof}

\begin{sloppypar} \printbibliography[title={References}] \end{sloppypar}

\end{document}